\documentclass[11pt,letterpaper]{amsart}
\usepackage{amsmath, amscd, amssymb}
\usepackage{mathrsfs}
\usepackage{pstricks}
\usepackage{color}
\usepackage{cancel}
\usepackage[mathscr]{eucal}
\usepackage{stmaryrd}
\usepackage[all]{xy}
\usepackage{enumitem}
\usepackage[backref]{hyperref}

\usepackage{pstricks}
\usepackage{cancel}
\usepackage{verbatim}
\usepackage{cite}

\numberwithin{equation}{section}

\newcommand{\CC}{\mathbb{C}}

\newcommand{\PP}{\mathbb{P}}
\newcommand{\QQ}{\mathbb{Q}}

\newcommand{\ZZ}{\mathbb{Z}}

\def\cln{\colon}

\def\cA{{\mathcal A}}

\newcommand{\cal}{\mathcal}

\def\cA{{\cal A}}

\def\cF{{\cal F}}

\def\cH{{\cal H}}

\def\cL{{\cal L}}
\def\cM{{\cal M}}

\def\cO{{\cal O}}
\def\cP{{\cal P}}

\def\cU{{\cal U}}

\def\begeq{\begin{equation}}
\def\endeq{\end{equation}}
\def\and{\quad{\rm and}\quad}

\def\and{\quad\text{and}\quad}

\newtheorem{prop}{Proposition}[section]
\newtheorem{theo}[prop]{Theorem}
\newtheorem{lemm}[prop]{Lemma}
\newtheorem{coro}[prop]{Corollary}

\newtheorem{conj}[prop]{Conjecture}

\newtheorem{defi-prop}[prop]{Definition-Proposition}
\newtheorem{quest}[prop]{Question}
\newtheorem{remark}[prop]{Remark}

\def\dbar{\overline{\partial}}

\def\beq{\begin{equation}}
\def\eeq{\end{equation}}

\def\bee{\begin{equation}}
\def\eeq{\end{equation}}

\title[]{Deformation rigidity for projective manifolds and isotriviality of smooth families
}

\thanks{The first author is supported by NSFC (No. 12271073 and 12271412), the second author is supported by NSFC (No. 12271073).}

\date{}

\author{Mu-Lin Li,~~~~~~~~Xiao-Lei Liu}

\begin{document}

\begin{abstract}Let $\pi\cln X\to \Delta^m$ be a proper smooth K\"ahler morphism from a complex manifold $X$ to the unit polydisc $\Delta^m$. Suppose the fibers over the complement of a proper analytic subset are biholomorphic to a fixed projective manifold $S$. If the canonical line bundle of $S$  is semiample, then we show that all fibers  over $\Delta^m$ are biholomorphic to $S$. As an application,  we obtain that for smooth families where the canonical line bundle of the generic fiber is semiample, birational isotriviality is equivalent to isotriviality. Moreover, we establish a new Parshin-Arakelov type isotriviality criterion.
\end{abstract}

\maketitle

\section{introduction}
\subsection{Deformation rigidity of projective manifolds}
In this paper, we study   {\it smooth families} of compact complex manifolds, which we define as proper smooth morphisms $\pi\cln X\to B$ between complex manifolds.  For each $t\in B$, denote by  $X_t=\pi^{-1}(t)$  the corresponding fiber of $\pi$.  Deformation rigidity of compact complex manifolds — which centers on the following question — has been studied since at least the work of   Bott \cite{Bo} and Kodaira-Spencer \cite{KS,KoS60}. 
\begin{quest}\label{quest}
Let $\pi\cln X\to \Delta=\{t\in\CC||t|<1\}$ be a smooth family of compact complex manifolds. Suppose that all fibers $X_t$ with $t\in\Delta^*=\Delta\setminus  \{0\}$  are biholomorphic to a fixed compact complex manifold $S$. When  is the central fiber $X_0$  also biholomorphic to $S$ (i.e., $X_0\cong S$)?
\end{quest}

We focus on the special case where  $S$ is a  projective manifold. From the perspective of algebraic geometry, projective manifolds $S$ admit a coarse classification via their  Kodaira dimension $\kappa(S)$. For projective manifolds of negative Kodaira dimension,  numerous results are available on deformation rigidity.  Notably,  it is well known that Hirzebruch surfaces are not deformation  rigid.
  For positive results, Siu \cite{Siu0}, Hwang-Mok \cite{HN1,HN2,HN3}  have made foundational contributions. For example, they established the deformation rigidity of  rational homogeneous spaces $S$  with the second Betti number $b_2(S)=1$, under the assumption that $\pi$ is a K\"ahler morphism. In what follows, we henceforth  assume that the morphism $\pi$ in Question \ref{quest} is K\"ahler.

 The algebraic formulation of deformation rigidity was investigated  by Matsusaka and Mumford in 1964, who focused on smooth polarized varieties.  A {\it polarized variety} is a pair $(S,\cH)$, consisting of a projective variety $S$ and an ample line bundle $\cH$ on $S$. Assume that all fibers of $\pi\colon X\to\Delta$ are smooth projective polarized varieties, and that all  fibers over $\Delta^*$ are isomorphic as polarized varieties. Then the first fundamental theorem (Theorem 2) of \cite{MM} shows that all fibers of $\pi$ are isomorphic provided that the generic fiber of $\pi$ is non-ruled.

  For projective manifolds with low Kodaira dimension, the first author established several deformation rigidity results in \cite{L1}. Specifically, if $S$ is an abelian variety or a principal bundle of abelian varieties  over a curve $C$ with genus $g(C)>1$, the first author proved in \cite[Theorem 1.2]{L1} that the deformation rigidity property above holds — meaning  $X_0\cong S$. Note that these manifolds satisfy $0\leq\kappa(S)\leq1$.

   However, to the best of our knowledge, there have been few studies addressing the deformation rigidity Question 1.1 for projective manifolds of higher Kodaira dimension — this is the central focus of the present paper.  We investigate deformation rigidity for projective manifolds $S$ equipped with  semiample canonical line bundle, which necessarily satisfy $\kappa(S)\geq0$. Our main result is stated below.
   
\begin{theo}\label{res1} 
 Let $\pi\cln X\to \Delta^m:=\{(z_1,\ldots,z_m)\in\CC^m||z_i|<1\}$ be a morphism satisfying the following property:
    
    (P):~$\pi$ is K\"ahler, and there exists a  domain $U$ in $\Delta^n$ obtained by removing a proper analytic subset such that, for every $t\in U$, the fiber $X_t$ is biholomorphic to a fixed projective manifold $S$. 
    
    If the canonical line bundle $\omega_S$ is semiample,  then every fiber of $\pi$ is biholomorphic to $S$.
\end{theo}

    
    

Before presenting the applications of this result, we offer several remarks.

Firstly, we want to compare Theorem \ref{res1} (for $m=1$) with some known results.

 On the one hand, we compare it to  the above mentioned Matsusaka-Mumford's theorem on smooth polarized varieties.  We note that the positivity assumption in our result is semiampleness while ampleness is needed in \cite[Theorem 2]{MM}. Moreover, we do not assume that the central fiber $X_0$  is projective polarized, nor that the isomorphisms between fibers over $\Delta^*$  preserve polarizations.  In particular, when the canonical line bundle of the generic fiber is semiample, our result is stronger than \cite[Theorem 2]{MM}.

On the other hand, we remark that Theorem \ref{res1} generalizes the main result (i.e., \cite[Theorem 1.2]{L1}) of \cite{L1}, where one knows that the canonical line bundle $\omega_S$ of $S$ is semiample.

Secondly, we want to give a remark comparing Theorem \ref{res1} to the non-separatedness of the moduli of  hyperk\"ahler manifolds.   

\begin{remark}
 It is shown (see \cite{Huy97,Huy99}) that any two birational hyperk\"ahler manifolds $Y$ and $Y'$ define non-separated points in the moduli space of  hyperk\"ahler manifolds. Equivalently,  there exist two  smooth families $\pi\cln X\to \Delta$ and  $\pi'\cln X'\to \Delta$ with $X_0\cong Y$ and $X_0'\cong Y'$ respectively,   such  that $\pi$ and $\pi'$ are isomorphic over $\Delta^*$. 
By contrast,  we emphasize that our setting imposes the additional requirement that  $\pi$ is a K\"ahler  morphism. 
\end{remark}

Finally, combining with the {abundance conjecture}, it is clear that we can obtain more results as applications of Theorem \ref{res1}.

 Recall that a projective manifold $S$ is  {\it minimal}, if  its canonical line bundle   $\omega_S$ is  nef. 
 A minimal projective manifold  $S$ is {\it good} if $\omega_S$ is semiample. The abundance conjecture  asserts that  every minimal projective manifold  is  good.

The abundance conjecture is classically known for curves and surfaces, whereas for threefolds it was a fantastic achievement (see  \cite{KMM}). In higher dimensions, the conjecture is also known to hold for manifolds of general type (see \cite{K0}). Consequently, for minimal manifolds $S$  in these cases, the canonical line bundle $\omega_S$ is semiample, and we immediately obtain the following corollary by applying Theorem \ref{res1}.

\begin{coro}\label{thmd23}
Let $\pi\cln X\to \Delta^m$ and $S$ satisfy the property $(P)$. If $S$ is minimal and satisfies that  either  $\dim S\le 3$ or $S$ is of general type, then  every fiber of $\pi$ is biholomorphic to $S$.
\end{coro}

This naturally leads to the following conjecture, which seems to be easier to carry out than the abundance conjecture:

\begin{conj} Let $\pi\cln X\to \Delta^m$ and $S$ satisfy the property $(P)$.  If $S$ is minimal with $\dim S\ge 4$,  then  every fiber of $\pi$ is biholomorphic to $S$.
\end{conj}

\subsection{Two types of isotriviality criteria of smooth families}
Throughout this paper, we always work over the complex number field $\CC$. 

We now present  some applications to the triviality of  smooth projective families $f\cln X\to B$. A family $f$ of varieties is {\it isotrivial} if all fibers of $f$ are isomorphic to each other.  We note that  isotriviality has different definitions in different papers.

 Let $B$ be a smooth projective curve and $\Sigma$ be a finite subset of $B$.  In 1962,  Shafarevich conjectured that the set of non-isotrivial families of smooth projective curves of genus $g\geq2$ over $B\setminus\Sigma$ is finite. Since then, isotriviality criteria for families have been a topic of enduring interest, with two primary types of criteria emerging from the Shafarevich conjecture, as outlined below.

{\bf 1). Hyperbolic isotriviality criterion}. One type of isotriviality criterion arises from the hyperbolicity part of the Shafarevich conjecture. The hyperbolic criterion says that:
if   $2g(B)-2+\#\Sigma\leq0$,  then  $f$ is isotrivial.
The Shafarevich conjecture, which later played an important role in Faltings' proof of the Mordell conjecture, was confirmed by Parshin \cite{Pa} for $\Sigma=\emptyset$   and by Arakelov \cite{Ar} in general.

   It is a natural and important question whether similar hyperbolic isotriviality criteria hold for families of higher dimensional varieties over higher dimensional bases. Because of the absence of minimal models among smooth varieties in higher dimensions, the isotriviality  becomes  birational isotriviality.

   A family $f\cln X\to B$ of varieties is called {\it birationally isotrivial} if there exists a dense open subset $B^0\subset B$,  a finite cover $B'\to B^0$ and a commutative diagram
 $$\xymatrix{
  X\times_BB' \ar@{-->}[rr]^{\Phi} \ar[dr]_{f_{B'}}
                &  &    B'\times X_0 \ar[dl]^{pr_1}    \\
                & B'                 }
                $$
with $\Phi$ is birational. In other words, denoting by $K=\CC(B)$ the function field of $B$, if the geometric generic fiber $X\times_B\overline K$ is birational, over $\overline K$, to $X_0\times_\CC\overline K$, then $f$ is called birationally isotrivial. Here $\overline K$ is the algebraic closure.
   It is known that a family of varieties is birationally isotrivial if all the fibers are birational to each other \cite{BB15}.

    Families of higher dimensional varieties over curves have been extensively  studied by several authors in recent years and the birational hyperbolic isotriviality criteria are now  well understood.  The strongest results known were obtained in \cite{VZ01, VZ02} and \cite{Kov02}; for a comprehensive overview, we refer the reader to survey articles \cite{Vi01} and \cite{Kov03}. For higher-dimensional bases, the birational hyperbolic isotriviality has also been investigated,  with \cite{WW23} providing an up-to-date summary of progress. 

A natural question is: when does birational isotriviality imply isotriviality? As an application of  Theorem \ref{res1}, we can show that for smooth families whose generic fiber is a good minimal manifold,   birational isotriviality is equivalent to  isotriviality.

\begin{theo}\label{g-gen-intr-b}(=Theorem \ref{g-gen-intr})
  Let $f\colon X\to B$ be a smooth projective family between smooth projective varieties. Suppose that the generic fiber of $f$  is good minimal. If $f$ is birationally isotrivial, then $f$ is isotrivial.
\end{theo}

Combining with the  well-understood birational  hyperbolic isotriviality criteria mentioned above  and  the known results on the abundance conjecture,  it is straightforward to derive a wealth of hyperbolic isotriviality conclusions. We do not aim to enumerate them here, but instead illustrate this with an example using \cite[Theorem 0.1]{VZ01}.

\begin{coro}
   Let $f\colon X\to B$ be a smooth projective family between smooth projective varieties where $B$ is a  curve. Suppose that  the generic fiber of $f$  is good minimal. If $g(B)\leq1$, then $f$ is isotrivial.
\end{coro}
Under the hypotheses of the corollary,  \cite[Theorem 0.1]{VZ01} shows that $f$  is birationally isotrivial, from which the corollary follows immediately by Theorem \ref{g-gen-intr-b}.

Remark that if the relative canonical line bundle $\omega_{X/B}$ of the smooth family $f$ is not ample, but the restriction of $\omega_{X/B}$ to every fiber is ample, then  the above result has been obtained by Kov\'acs \cite[Theorem 2.16]{Kov96}, which relies on the first fundamental theorem (Theorem 2) of \cite{MM}.

{\bf 2). Parshin-Arakelov isotriviality criterion}. The other type of isotriviality criterion arises from Parshin-Arakelov's work on the Shafarevich conjecture.  It can be stated as follows:
\begin{quote}
Let $f\colon X\to B$ be a family of curves of genus $g\geq2$ where $B$ is a smooth projective curve. Suppose that  no fiber contains a $(-1)$-curve.  If  $\deg(\det f_*\omega_{X/B})=0$,  then $f$ is isotrivial.
\end{quote}

An interesting observation here is that $\deg(\det f_*\omega_{X/B})=0$ if and only if $\deg(\det f_*\omega_{X/B}^{\otimes k})=0$ for some $k\geq1$. By the theory of fibered surfaces, this is further equivalent to   $\deg(\det f_*\omega_{X/B}^{\otimes k})=0$ for all $k\geq1$.

Despite its importance, this type of isotriviality criterion has seen very little generalization to date (see, for example, \cite{LX}).  In order to obtain a further understanding,   we  present the following generalized Parshin-Arakelov isotriviality criteria for smooth families, see Corollaries \ref{g-gen-2} and \ref{g-gen-intr-coro}.

\begin{theo}\label{g-gen}
  Let $f\colon X\to B$ be a smooth projective family between smooth projective varieties,  and suppose the generic fiber $F$ of $f$ is minimal. Assume that one of the following two conditions holds:

  i) either $F$  is of general type or $\dim F\leq 3$, and $\deg(\det f_*\omega_{X/B}^{\otimes k})=0$ for all $k\geq2$;

  ii) $B$ is a curve,   $F$  is of general type and $\deg(\det f_*\omega_{X/B}^{\otimes k})=0$ for some sufficiently large $k$.

 \noindent Then $f$ is isotrivial.

\end{theo}

As with  the well-studied  hyperbolic isotriviality criterion,  many cases of the Parshin-Arakelov isotriviality criterion are still open.

\subsection{Ideas of the proof of Theorem \ref{res1}} The main idea of the proof of our main result, Theorem \ref{res1},  is to use the separatedness of the moduli space $P_h$ of polarized manifolds.
Our major task is  the following two parts: one is to construct a $\pi$-ample line bundle on $X$, and the other is to show that the canonical line bundle $\omega_{X_0}$ of $X_0$ is semiample.

In the following  explanation, we not only outline our proof, but also interpret the difficulties we encountered. The innovations to overcome these difficulties are also mentioned.

{\bf(1) New extension of line bundles}

The desired $\pi$-ample line bundle on $X$ is constructed via extension. 
 In contrast to standard extension procedures, we introduce a new kind of extension. Presicely,  we extend a line bundle $\cL_0$ on the central fiber $X_0$ to a global one $\cL$ over $X$. Moreover, if $\cL_0$ is ample, then the extended line bundle $\cL$ over $X$ is $\pi$-ample. This construction relies on the theory of variations of the Hodge structures.

{\bf (2) Semiampleness of $\omega_{X_0}$}

We need to deduce the semiampleness of $\omega_{X_0}$ from that of $\omega_{X_t}$ for each $t\neq0$. However, this implication does not hold for an arbitrary  line bundle $\cL$ on $X$, even if $\cL_t:=\cL|_{X_t}$ is very ample on $X_t\cong S$. In fact, establishing the  positivity of $\omega_{X_0}$ is the most challenging  part of this paper. Fortunately, we only need positivities of the pluricanonical and the adjoint linear systems on $X_0$.

Let $\cL$ be a $\pi$-ample line bundle over $X$. We can reduce the  proof of the semiampleness of $\omega_{X_0}$   to that of the nefness of $\omega^{\otimes m}_{X_0}\otimes \mathcal{L}_0$ for each $m\in\mathbb{N}$, by Kawamata's work on the theory of numerical Kodaira dimension.   To verify this nefness,  we  employ an  analytic approach, checking the  metric characterization of nefness. Roughly speaking, we use the theory of deformations of complex structures to construct sufficiently  many  $2$-forms on $X_0$, which have some positivity on $X_0$.  The extension of the nefness property then follows from the known  criterion of nefness on projective manifolds.






 The organization of this paper is as follows.

 In \S\ref{sec2},  we give notations and basics of moduli of polarized manifolds which will be used in our proofs.
 In \S\ref{sec3}, we give our new extension of line bundles and construct $\pi$-ample line bundles on $X$. 
 Since the semiampleness of $\omega_{X_0}$ is the most difficult part of this paper, we first prove some positivity on the central fiber in \S\ref{sec4n} as a preparation step.  
 The proof of our main result, Theorem \ref{res1}, is given in \S\ref{sec5n}. 
 In the final  section, \S\ref{sec6n}, we provide our application to the isotriviality criteria.

{\bf Acknowledgement:}  The authors would like to thank Kang Zuo and Sheng Rao  for their interest and useful discussions. They would like to thank Feng Hao for introducing them to the results of S\'andor J. Kov\'acs. They are grateful to Sheng-Li Tan, Tong Zhang, Yong Hu, Wei Xia, and Ariyan Javanpeykar  for their conversations on isotriviality criteria. Finally, the authors would like to express their sincere gratitude to
 the referees for their useful detailed suggestions, which have made the paper more  valuable and readable.

\section{Preliminaries}\label{sec2}
\subsection{Setup and notation}
 Let $S$ be a compact complex manifold. Throughout the following, we always assume that $\dim S=n$ without mention.

  Let $\Omega_S$ be the cotangent bundle of $S$, $\Omega^p_S=\wedge^p\Omega_S$, and $\omega_S=\Omega^n_S$ be the  canonical line bundle of $S$. The {\it Kodaira dimension} $\kappa(S)$ of $S$ is defined as
$$
\kappa(S)=\max\left\{k\in\mathbb{Z}_{\ge0}|\limsup\limits_{m\to \infty}\frac{h^0(S,\omega_S^{\otimes m})}{m^k}\right\},
$$
and the {\it numerical Kodaira dimension} $\nu(S)$ is  defined as
$$\nu(S)= \max\{k\in\mathbb{Z}_{\ge0}|c^k_1(\omega_S)\neq 0\in H^{2k}(S,\CC)\}.$$
It is known (see \cite{K1}) that, if $\omega_S$ is nef, then
\beq\label{eqndim}
\kappa(S)\leq\nu(S),
\eeq
and the equality holds if and only if  $\omega_S$ is semiample.

Let $\varphi\cln X\to B$ be a proper morphism of complex manifolds. Let $\cL$ be a line bundle on $X$, and $\cL_t=\cL|_{X_t}$ for any $t\in B$. We recall the following classical definitions for the reader's convenience:
\begin{itemize}

\item{$\cL$ is {\it $\varphi$-free} if  the natural morphism $\varphi^*\varphi_*\cL\to \cL$ is surjective. In particular, if $\cL$ is $\varphi$-free then it induces a morphism $X\to \mathbb{P}(\varphi_*\cL)$ over $B$;}

\item{$\cL$ is {\it $\varphi$-very ample} if $\cL$ is $\varphi$-free and the induced morphism $X\to \mathbb{P}(\varphi_*\cL)$ is a closed immersion;}

\item{ $\cL$ is {\it $\varphi$-semiample} (resp. {\it $\varphi$-ample}) if  $\cL^{\otimes k}$ is $\varphi$-free  (resp. $\varphi$-very ample) for some positive integer $k$.}
\end{itemize}

A smooth $d$-closed $(1,1)$-form $\alpha$ on $X$ is called a {\it relative K\"ahler form} for $\varphi$, if there exists an open cover $\{U_{\lambda}\}$ of $B$ and K\"ahler forms $\eta_{\lambda}$ on $U_{\lambda}$ such that $\alpha|_{\varphi^{-1}(U_\lambda)}+\varphi^*\eta_\lambda$ are K\"ahler forms on $\varphi^{-1}(U_\lambda)$. A morphism $\varphi$ is said to be {\it K\"ahler} \cite[Definition 4.1]{F1}, if there exists a relative K\"ahler form for $\varphi$.

 In what follows,  we will always assume that families $\pi$ satisfy the following condition  unless explicitly stated otherwise:

 \medskip
\begin{description}
    \item[{\underline{\boldmath{$(*)$}}} ] $\pi\cln X\to\Delta$ is a proper smooth K\"ahler morphism between complex manifolds, with $X_t\cong S$ for all $t\neq0$, where $S$ is a fixed projective manifold.
\end{description}

\subsection{Moduli of polarized manifolds} For a polynomial $h\in\QQ[T]$, let $\cP_h$ be the fibered category over the category of schemes, such that for a scheme $B$ the groupoid $\cP_h(B)$ is defined as
 \begin{align*}
  \cP_h(B)=&\bigg\{(f\cln X\to B,\cH)|f\, \mbox{is a smooth morphism};\,  \cH \, \mbox{is a line bundle on $X$;}\\
&(X_b,\cH_b)\, \mbox{is a polarized manifold,\, $\omega_{X_b}$ is semiample, $\forall b\in B$};\\
&h(m)=\chi(\cH^{\otimes m}), \mbox{ $\forall m\in \ZZ$}\bigg\}.
\end{align*}

 The arrow of $\cP_h(B)$ from $(f\cln X\to B,\cH)$ to $(f'\cln X'\to B,\cH')$ is an isomorphism $\tau\cln  X\to X'$ such that $ \cH_b$ and $ (\tau^*\cH')|_{X_b}$ are numerical equivalent for all $b\in B$. By \cite[Theorem 4.6]{Vi0}, there exists a coarse separated moduli algebraic space $P_h$ for $\mathcal{P}_h$. Note that $P_h$ is the quotient of the moduli space of polarized manifolds by compact equivalence relations (see \cite{Vi0}).

If $(\varphi\cln X\to B,\cH)$ is a smooth family of polarized manifolds with Hilbert polynomial $h(m)=\chi(\cH^{\otimes m})$ for all $m\in\ZZ$, then we have  an induced morphism
$$\mu_{\varphi}\cln B\to P_h$$
satisfying that $\mu_{\varphi}(b)=(X_b,\cH_b)\in P_h$ for all $b\in B$.

\section{Extension of line bundles}\label{sec3}

First, we recall some basics of variations of Hodge structures from \cite[Chapter 10]{Voi}, \cite[Section II.1]{G68} and \cite[Chapter 4]{CMP}.
 When $\pi\cln X\to \Delta$ is smooth, by \cite[Theorem 9.3]{Voi} there exists a diffeomorphism
\beq
T\cln X\to X_0\times \Delta
\eeq
such that $p_2\circ T=\pi$ and $T_0|_{X_0}=\mathrm{Id}_{X_0}\cln X_0\to X_0$, where $p_i~(i=1,2)$ denotes the $i$-th projection from $X_0\times \Delta$ and $T_0:=p_1\circ T\cln X\to X_0$. Let
$$\kappa_t=T_0|_{X_t}\cln X_t\to X_0$$
be the restriction of $T_0$ on $X_t$. Then $\kappa_t$ is a diffeomorphism and we denote by $\psi_t$ its inverse diffeomorphism, as illustrated in the following commutative diagram.
\begin{eqnarray*}
\xymatrix{
X_t \ar[r] \ar@<.5ex>[rd]^{\kappa_t}  &  X\ar[d]^{T_0}\ar[r]^(0.35){T}\ar[rd]^(0.35){\pi}|\hole  & X_0\times\Delta \ar[dl]^(0.6){p_1}\ar[d]^{p_2}\\
& X_0\ar@<.5ex>[lu]^{\psi_t}  & \Delta
}
\end{eqnarray*}

For each $t\in\Delta$, by using the Hodge decompostion 
\beq\label{iden-HD}
 H^2(X_t,\CC)=\bigoplus_{r,s\ge0,r+s=2}H^{r,s}(X_t,\CC),
\eeq
we get the Hodge filtration 
$$F^kH^{2}(X_t,\mathbb{C}):=\bigoplus_{2\ge r\ge k}H^{r,2-r}(X,\CC)~~~~(2\ge k\ge0)$$ 
on the complex vector space $H^{2}(X_t,\mathbb{C})$ 
  which satisfies  
\begin{equation}\label{identity-3}
  H^{r,s}(X_t,\CC)=  F^rH^{2}(X_t,\mathbb{C})\cap \overline{F^sH^{2}(X_t,\mathbb{C})}.
\end{equation}
So, by the equality
\beq\label{iden-1}
\psi_t^*H^{2}(X_t,\CC)=H^{2}(X_0,\CC),
\eeq
we can define  a Hodge filtration on the space $H^{2}(X_0,\CC)$ as follows:
\beq\nonumber
F^k_t:=\psi_t^*F^kH^{2}(X_t,\mathbb{C}).
\eeq
Let $b^{k}=\dim F^k_t$, then the subspaces $F^k_t$ induce  the following holomorphic morphism as in \cite[(1.1) Theorem]{G68} or \cite[Theorem 10.9]{Voi}
$$ \cP^{k}\cln  \Delta\to \mathrm{Gr}(b^{k},H^2(X_0,\CC)),~~~~\cP^{k}(t)=F^k_t,$$
where $\mathrm{Gr}(b^{k},H^2(X_0,\CC))$ is the  Grassmannian of $b^{k}$-dimensional vector subspaces of $H^2(X_0,\CC)$.

\begin{lemm}\label{lemm4}
Let $\pi\cln X\to \Delta$ satisfy $(*)$ and $\cA$ be a line bundle  on $X_0$, then there exists a line bundle $\cL$ on $X$ such that $c_1(\cL|_{X_0})=c_1(\cA)\in H^2(X_0,\CC)$.
\end{lemm}
\begin{proof}
Since  $\pi$ is a proper smooth K\"ahler morphism, we can assume that $X$ is K\"ahler after shrinking the disc small enough. So  $\dim H^{p,q}(X_t,\CC)$ is independent of the choice of $t\in \Delta$. By \cite[Theorem 1.0.3]{Ba1}  (see also \cite[Theorem 1.2]{P3} or \cite[Theorem 1.4]{RT1}), $X_0$ is Moishezon,  and thus   $X_0$ is projective for  it is also K\"ahler.

{\bf Step 1}. First we want to show that the morphism $\cP^{k}$ is constant.

Let $\tilde{\pi}=\pi|_{{\pi}^{-1}(\Delta^*)}\cln \pi^{-1}(\Delta^*)\to \Delta^*$ be the restriction of $\pi$ on $ \pi^{-1}(\Delta^*)$. Since $X_t\cong S$ for all $t\neq 0$, the smooth family $\tilde{\pi}$ is locally trivial by \cite{FG}. Thus, for each $t\in\Delta^*$, there exists a small neighborhood $U_t\subset\Delta^*$ of $t$ such that
$$\tilde{\pi}^{-1}(U_t)\cong U_t\times S.$$
It follows that the Kodaira-Spencer map for $\pi|_{\tilde{\pi}^{-1}(U_t)}$
$$\rho_t\cln T_{U_t,t}\to H^1(X_t,T_{X_t})$$
is zero, and then the differential $d\cP^{k}$ vanishes on $\Delta^*$ by \cite[(1.23) Theorem]{G68} or \cite[Theorem 10.4]{Voi}. Consequently, we have that $ \cP^{k}$ is constant on $\Delta$.
Therefore
\begin{eqnarray*}
    \psi_t^*F^kH^{2}(X_t,\mathbb{C})=F^kH^{2}(X_0,\mathbb{C}) 
\end{eqnarray*}
for $0\le k\le 2$.
Then, by the formula (\ref{identity-3}), we have
\beq\label{iden-2}
\psi_t^*H^{r,s}(X_t,\CC)=H^{r,s}(X_0,\CC)
\eeq
for $t\in \Delta$ and $0\le r,s\le 2$ with  $r+s=2$.

{\bf Step 2}. We describe the desired line bundle locally as $\kappa_t^*c_1(\cA)$.

By the exponential sequence, we have the following exact sequence
\beq\label{eqnexp-1}
H^1(X_t,\cO_{X_t}^*)\stackrel{c_1}{\longrightarrow}H^2(X_t,\ZZ)
\stackrel{\iota_t}{\longrightarrow}H^2(X_t,\cO_{X_t})\to.
\eeq
Moreover, by the Hodge decomposition \eqref{iden-HD},
the map $\iota_t$ is the composition of the following maps
\beq
H^2(X_t,\ZZ)\hookrightarrow H^2(X_t,\CC)\to H^{0,2}(X_t,\CC)\cong H^2(X_t,\cO_{X_t}).
\eeq
Combining with \eqref{iden-1} and (\ref{iden-2}), we see that
$$\iota_0=\psi_t^* \circ\iota_t\circ\kappa_t^*.$$
Therefore  we have that
\beq\label{iden-2-0}
\iota_t(\kappa_t^*c_1(\cA))=(\kappa_t^*\circ\iota_0\circ\psi_t^*)(\kappa_t^*c_1(\cA))
=\kappa_t^*\big(\iota_0(c_1(\cA))\big)=0
\eeq
for each $t\in\Delta$, since $\iota_0(c_1(\cA))=0$ by \eqref{eqnexp-1}.

{\bf Step 3}.  We show that our constructed line bundle corresponds to $T_0^*c_1(\cA)$.

Consider the exponential sequence for $X$
\beq\label{eqnexp-2}
H^1(X,\cO_{X}^*)\stackrel{c_1}{\longrightarrow}H^2(X,\ZZ)
\stackrel{\iota}{\longrightarrow}H^2(X,\cO_{X})\to.
\eeq
Applying the Leray spectral sequence to $\pi$ for the sheaves $\ZZ$ and $\cO_X$, we have the following commutative diagram
 $$
\begin{CD}
H^1(X,\cO_{X}^*)@>^{c_1}>>H^2(X,\ZZ)@>^{\iota}>>H^2(X,\cO_{X})@>>>
\\
@.@V^{\cong}VV @V^{\cong}VV \\
@.H^0(\Delta, R^2\pi_*\ZZ)@>^{\sigma}>>H^0(\Delta,R^2\pi_*\cO_X)@>>>.
\end{CD}
$$
Similarly, by the base change theorem, for each $t\in\Delta$,
we have the following equality
$$\sigma_t=\iota_t\cln H^2(X_t,\ZZ)\to H^2(X_t,\cO_{X_t}).$$
By \eqref{iden-2-0}
$$\sigma(T_0^*c_1(\cA))(t)=\sigma_t(\kappa_t^*c_1(\cA))=\iota_t(\kappa_t^*c_1(\cA))=0$$ for all $t\in\Delta$, so we have
\beq
\iota(T_0^*c_1(\cA))=\sigma(T_0^*c_1(\cA))=0.\nonumber
\eeq
Therefore, by \eqref{eqnexp-2}, there exists a line bundle $\cL$ on $X$ such that the first Chern class of the restriction $\cL|_{X_0}$ satisfies that $c_1(\cL|_{X_0})=c_1(\cA)$.

 \end{proof}

\begin{prop}\label{proj}Let $\pi\cln X\to \Delta$ satisfy $(*)$, then $\pi$ is a projective morphism.
\end{prop}
\begin{proof}As shown in the first paragraph of the proof of Lemma \ref{lemm4}, $X_0$ is projective. So we can take  an ample line bundle $\cA$  on $X_0$. By Lemma \ref{lemm4}, there exists a line bundle $\cL$ on $X$ satisfying $c_1(\cL|_{X_0})=c_1(\cA)$. Let $\cL_t= \cL|_{X_t}$ be the restriction of $\cL$ on $X_t$ for all $t\in\Delta$, then $\cL_0$ is an ample line bundle on $X_0$. For a sufficiently large positive number $m\in\mathbb{N}$, we have $H^{i}(X_0,\cL_0^{\otimes m})=0$ for $i\ge1$ and $\cL_0^{\otimes m}$ is very ample. By \cite[Theorem 4.2, Chapter IV]{BanS}, there exists a neighborhood $U\subset\Delta$ of $0$ such that $\cL_t^{\otimes m}$ is very ample for each $t\in U$. This implies that $\cL_t$ is ample for each $t\in \Delta$,  because $X_t\cong S$  for $t\neq0$ and the ampleness of $\cL_t$ only depends on its numerical equivalence class. Then there exists a sufficiently large positive number $\bar{m}\in\mathbb{N}$ such that $H^{i}(X_t,\cL_t^{\otimes \bar m})=0$ for $i\ge1$ and $\cL^{\otimes \bar m}_t$ are very ample for all $t\in\Delta$. Therefore the morphism
\beq
\pi^*\pi_*\cL^{{\otimes \bar m}}\to \cL^{{\otimes \bar m}}\nonumber
\eeq
induces an embedding $X\to \mbox{Proj}\left(\pi_*\cL^{{\otimes \bar m}}\right)$. Thus $\pi$ is a projective morphism.
 
\end{proof}

\section{Nefness over the central fiber}\label{sec4n}

In this section, we prove a special case of nefness over the central fiber   $X_0$.

A line bundle $\cL$ on $X$ is {\it nef} if $\deg(\cL|_C)\geq0$ for every irreducible curve $C\subset X$. In the following
lemma, we will use this definition to assure the nefness.

\begin{lemm}\label{form}
  Let $\pi\cln X\to \Delta$ satisfy $(*)$ and  $\cL$ be a line bundle on $X$.   Suppose furthermore that $\pi$ is projective and
 \begin{enumerate}[label=(\roman*).]
  \item{ $\cL_s$ is very ample for each $s\neq 0$; }
\item{  $h^0(X_s,\cL_s)=h^0(X_0,\cL_0)$ for each $s\neq 0$.}
    \end{enumerate}
  Then the line bundle $\cL_0$ on $X_0$ is  nef.
\end{lemm}
\begin{proof}
Let
$$
\phi\cln  X\dashrightarrow \mathbb{P}(\pi_* \mathcal{L})
$$
 be the  rational map over $\Delta$ induced by the natural morphism  $\pi^*\pi_* \mathcal{L}\to \mathcal{L}.$
We may assume that $\phi$ is set-theoretically well defined outside an analytic set $Z$
with $\mathrm{codim}(Z,X)\geq2$. 
 Since $\pi$ is flat, we know that $\pi_*\mathcal{L}$ is locally free by the assumption (ii).
 Thus we can assume $\PP(\pi_*\mathcal{L})=\PP^N\times \Delta$ after shrinking the disc small enough, and  the restriction 
 $$\phi|_{X\setminus Z}\cln X\setminus Z\to \PP^N\times \Delta$$ 
 of $\phi$  is a morphism.  
 By our assumption (i),  for each $s\neq 0$, the restriction
$$\phi_s\cln X_s\to  \mathbb{P}^N$$
of $\phi$ on $X_s$ is an embedding defined by the complete linear system $|\mathcal{L}_s|$. The restriction of $\phi$ on $X_0$
$$\phi_0\cln X_0\dashrightarrow  \mathbb{P}^N$$
is a rational map defined by  sections of $\pi_*\mathcal{L}\otimes \mathbb{C}(0)$ on $X_0$, and it is easy to see that $Z$ is contained in $X_0$.

The pullback line bundle ${\cF}:=\phi_{X\setminus Z}^*p^*_1\cO_{\PP^N}(1)$ on ${X\setminus Z}$ satisfies that ${\cF}|_{\pi^{-1}(\Delta^*)}\cong \cL|_{\pi^{-1}(\Delta^*)}$, where  $p_{1}\cln \PP^N\times \Delta\to \PP^N$ is the first projection. By \cite[Example 4.3.12]{Huy},  there exists a natural metric $\tilde{h}$ on $\cO_{\PP^N}(1)$ such that $\frac{\sqrt{-1}}{2\pi}\dbar\partial\log \tilde{h}$ is the Fubini-Study form $\omega_{\mathrm{FS}}$. We denote  the pullback metric of $\tilde h$ on ${\cF}$ by $h={\phi_{X\setminus Z}^*}p_1^*\tilde{h}$.

We will use the trivialization of $\pi$ in the following proof, please see the corresponding notations in Section \ref{sec3}. 
Now we will check the definition of the nefness of $\cL_0$  in four steps.  

{\bf Step 1.} In this step, we will construct $2$-forms  $\gamma_s~(s\in\Delta^*)$ on $X_0$ representing  the class $c_1(\cL_{0})\in H^{2}(X_0,\CC)$.

 Denote by  $h_s$  the restricted metric of $h$ on $\cL_s$, then for each $s\in\Delta^*$ the Chern class
\beq
c_1(\cL_{s})=\left[\frac{\sqrt{-1}}{2\pi}\dbar\partial\log h_s\right]\in  H^{1,1}(X_s,\CC)\nonumber
\eeq
and
\beq\label{inequ-1}
\frac{\sqrt{-1}}{2\pi}\dbar\partial\log h_s=\phi_s^*\omega_{\mathrm{FS}}.
\eeq
 Note that $\frac{\sqrt{-1}}{2\pi}\dbar\partial\log h_s$ is a positive (1,1)-form on $X_s$.

By the pullback formula of the Chern classes, we have
\beq\label{eqnchern}
c_1(T^{-1*}\cL)=T^{-1*}(c_1(\cL)).
\eeq
Let $\tilde{\iota}_s\cln X_0\to X_0\times\Delta$ be the morphism defined by $\tilde{\iota}_s(z)=(z,s)$ for all $z\in X_0$, and let $\iota_s\cln X_s\to X$ denote the inclusion morphism induced by  the submanifold structure of $X_s\subset X$. Then we have the following commutative diagram
\begin{eqnarray*}
\xymatrix{
X_s \ar[r]^{\iota_s}   &  X\ar[d]^{T_0} & X_0\times\Delta\ar[l]_(0.55){T^{-1}}  \\
& X_0\ar@<1ex>[lu]^{\psi_s} \ar[ru]_(0.45){\tilde\iota_s}
}
\end{eqnarray*}
and we know that $\iota_s\circ\psi_s=T^{-1}\circ \tilde{\iota}_s$. Therefore, by \eqref{eqnchern},
\beq
\tilde{\iota}_s^*\left(c_1(T^{-1*}\cL)\right)=\tilde{\iota}_s^*T^{-1*}(c_1(\cL))=\psi_s^*(c_1(\cL_s)).
\eeq
Since the disc $\Delta$ is contractible, the bundles  $\tilde{\iota}_s^*\left(T^{-1*}\cL\right)$ for $s\in \Delta$ are mutually isomorphic as topological bundles.  Hence their first Chern classes coincide, so that
\beq
\psi_s^*\left(c_1(\cL_s)\right)=c_1(\cL_0)\in H^2(X_0,\CC),\nonumber
\eeq
and
$$\left[\frac{\sqrt{-1}}{2\pi}\psi_s^*\left(\dbar\partial\log h_s\right)\right]=\psi_s^*\left[\frac{\sqrt{-1}}{2\pi}\dbar\partial\log h_s\right]\in H^2(X_0,\CC).$$
Thus 
$$\gamma_s:=\frac{\sqrt{-1}}{2\pi}\psi_s^*\left(\dbar\partial\log h_s\right)$$ represents the class $c_1(\cL_0)\in H^{2}(X_0,\CC)$.

{\bf Step 2.} In this step we will calculate the explicit expression of the forms $\gamma_s$ using the coordinates of $X$ locally.

 As in \cite[\S5.3(b)]{KS86}, possibly after shrinking the disc small enough, we can assume that $X$ has  the following finite open cover 
 $$X=\bigcup_{k}\cU_k=\bigcup_{k}(V_k\times\Delta)$$
where $X_0=\cup_k V_k$ and $\cL|_{\cU_k}$ is trivial.  
Let $(\xi^1_k,\cdots,\xi_k^n, s)$ be the holomorphic coordinates of $\cU_k$. Suppose that $\cU_j\cap \cU_k\neq \emptyset$, then the coordinate transformation
\beq
(\xi^1_j,\cdots,\xi_j^n,s)=(g^1_{jk}(\xi_k,s),\cdots,g^n_{jk}(\xi_k,s), s)\nonumber
\eeq
satisfies that all $g^l_{jk}(\xi_k,s)$ are holomorphic functions of $\xi_k^1,\ldots,\xi_k^n,s$. Note that $X$ can be considered as a complex manifold with the complex structure defined on the smooth manifold $X_0\times \Delta$ through the diffeomorphism $T^{-1}\cln X_0\times\Delta \to X$. For any point $(z,s)\in X_0\times \Delta$ with $T^{-1}(z,s)\in \cU_j$, we let $T^{-1}(z,s)=(\xi_j,s)$. Locally we can write
\beq
T^{-1}(z,s)=(\xi_j^1(z,s),\cdots , \xi_j^n(z,s),s),\nonumber
\eeq
where $\xi_j^l(z,s)$ are smooth functions. These functions $(\xi_j^l(z,s))$ can be considered as the local complex coordinates of $X_s$. Let $(z^1,\ldots,z^n)$ be arbitrary local complex coordinates of a point $z$ of $X_0$, then
\beq\label{eqnvhs1}
\xi_j^{l}(z,s)=\xi_j^{l}(z^1,\ldots,z^n,s),~~l=1,\ldots,n
\eeq
are smooth functions of the complex variables $z^1,\ldots,z^n,s$.  For any fixed point $s\in\Delta^*$, $X_s$ is the complex structure of the smooth manifold $X_0$ by the system of local complex coordinates:
\beq\label{eqnvhs2}
\{z\to \xi_j(z,s)|j=1,2,\ldots\},~~\xi_j(z,s)=(\xi_j^1(z,s),\ldots,\xi_j^n(z,s)).
\eeq
Here by ``the smooth manifold $X_0$" we mean the underlying differential manifold of the complex manifold $X_0$.

Let $\cU$ be any element in $\{\cU_k\}_k$.  After shrinking the disc small enough and taking a finer open cover  $\{\cU_k\}_k$ possibly, we may assume that  the holomorphic map
$$\phi|_{\cU\setminus Z}\cln \cU\setminus Z\to \PP^N\times\Delta$$
sends $(\xi,s)=((\xi^1,\ldots,\xi^n,)^t,s)$ to $([1:w^1:\cdots:w^{N}],s)$.  Since $\mathrm{codim}(Z\cap 
\cU, \cU)\geq 2$, the Hartogs’ extension theorem of holomorphic functions implies that the local functions $w^i$ are  well defined on $\cU$.

For each $s\in\Delta$,  let $U_s=\cU\cap X_s\subset X_s$. We still use the notation $\phi_s$ to denote the restriction morphism of $\phi|_{\cU\setminus Z}$ on $X_s$ for simplicity, and let $\phi_s$ send $\xi=(\xi^1,\ldots,\xi^n)^t$ to $[1:w^1_s:\cdots:w^{N}_s]$. 
Let  $w_s=(w_s^1,\ldots,w_s^N)^t$ and $\|w_s\|^2=\sum_k|w^k_s|^2$. The Fubini-Study form $\omega_{\mathrm{FS}}$ on  $\phi_s(U_s)\subset\PP^N$ is 
\begin{eqnarray*}
\frac{\sqrt{-1}}{2\pi}\dbar\partial\log h_s&=&\frac{\sqrt{-1}}{2\pi}\cdot\frac{1}{(1+\|w_s\|^2)^2}\sum_{l,m}h_{lm}d w^l_s\wedge d\bar{w}^m_s,
\end{eqnarray*}
where $h_{lm}=(1+\|w_s\|^2)\delta_{lm}-\bar{w}^l_sw^m_s$. 
Then the following induced form on $U_s$ satisfies that
\begin{eqnarray}\label{eqn-bound}
\begin{split}
\sum_{i,j}f_{ij}d\xi^i\wedge d \bar{\xi}^j 
:=&(1+\|w_s\|^2)^2\cdot\dbar\partial\log h_s\\
=&\sum\limits_{i,j,l,m} \frac{\partial w^l_s}{\partial \xi^i}h_{lm}\frac{ \partial\bar{w}^m_s}{\dbar\bar{\xi}^j}d\xi^i\wedge d \bar{\xi}^j,
\end{split}
\end{eqnarray}
since $dw_s^l=\sum_i \frac{\partial w^l_s}{\partial \xi^i}d\xi^i$. We remark that $w_s^l=w_s^l(\xi,s)$, $f_{ij}=f_{ij}(\xi,s)$ and $\frac{\partial w^l_s}{\partial \xi^i}=\frac{\partial w^l_s}{\partial \xi^i}(\xi,s)$  are all smooth functions of $\xi^1,\ldots,\xi^n,s$ here.

In order to calculate the 2-forms $\gamma_s$ on $X_0$ locally, we apply the deformation theory of complex structures (see \cite[\S5.3(b)]{KS86}) to the above form.  
Fixing the arbitrarily selected point $s\in\Delta^*$, by \cite[Lemma 2.5]{RZ} and \eqref{eqnvhs1},
\beq\label{dfm-stru}
d\xi^i(z,s)=\sum_r \varphi_r^i d{z}^r+\sum_k\bar\lambda_k^id\bar{z}^k,~~d\bar{\xi}^j(z,s)=\sum_t \bar\varphi_t^j d\bar{z}^t+\sum_l\lambda_l^jdz^l,
\eeq
where $\varphi_r^j(z,s),\lambda_k^j(z,s)$ are smooth on $s$ and $\lambda_k^j(z,0)=0$  for any $z\in X_0$.   Thus
\begin{align*}
\psi_s^*\left(f_{ij}d\xi^i\wedge d\bar{\xi}^j\right)
=&f_{ij}\bigg(\sum_r {\varphi}_r^idz^r+\sum_k\bar{\lambda}_k^id\bar{z}^k\bigg)\wedge \bigg(\sum_t \bar\varphi_t^j d\bar{z}^t+\sum_l{\lambda}_l^jdz^l\bigg)\\
=&\sum_{r,t} {\varphi}_r^if_{ij}\bar\varphi_t^jdz^r\wedge d\bar{z}^t+\sum_{k,t} \bar{\lambda}_k^if_{ij}\bar\varphi_t^jd\bar{z}^k\wedge d\bar{z}^t\\
&+\sum_{r,l} {\varphi}_r^i f_{ij}{\lambda}_l^jdz^r\wedge dz^l-\sum_{k,l} {\lambda}_l^j f_{ij}\bar{\lambda}_k^i dz^l\wedge d\bar{z}^k.
\end{align*}
Here we still use the notation $f_{ij}$ to denote its pullback by $\psi^*_s$ for simplicity, that is, we regard $f_{ij}=f_{ij}(\xi(z,s),s)$  as a function  of $z^1,\ldots,z^n,s$.  We deal with the notations $w_s^l,h_{lm}$ and $\frac{\partial w^l_s}{\partial \xi^i}$ similarly in the following context. 
By Step 1, we know that the $(1,1)$-component $(\gamma_s)_{1,1}$ of $\gamma_s$ is 
\begin{align*}
\frac{2\pi}{\sqrt{-1}}\cdot (1+\|w_s\|^2)^2\cdot (\gamma_s)_{1,1}&=\left(\psi_s^*\left((1+\|w_s\|^2)^2\cdot\dbar\partial\log h_s\right)\right)_{1,1}\\&=\sum\limits_{i,j,k,l}\left({\varphi}_i^l f_{lk}\bar{\varphi}_j^k- {\lambda}_i^k f_{lk}\bar{\lambda}_j^l\right) dz^i\wedge d\bar{z}^j.
\end{align*}

  {\bf Step 3.} In this step we will show that  the form $(\gamma_s)_{1,1}$ has some positivity locally on $X_0$ when $s$ is small enough.

Let  $H=(h_{lm})$, then $H=(1+\|w_s\|^2)E-\bar w_sw_s^t$ where $E$ is the unit matrix.  Let $\Phi=(\varphi^k_i),~\Lambda=(\lambda^k_i)$ and $J=(\frac{\partial w^j_s}{\partial \xi^i})$, then $$F:=(f_{ij})=JH\bar J^t=(1+\|w_s\|^2)J\bar J^t- J\bar w_sw_s^t\bar J^t.$$
Since $\lambda_k^j(z,0)=0$ for any $z\in X_0$, we know that the 2-norm $\|\Lambda\|_2$ of the matrix $\Lambda$   tends to 0 as $s$ tends to 0.

 For any $x=(x_1,\ldots,x_n)^t\neq0$,
\begin{align*}
  &\bar x^t(\Phi F \bar\Phi^t-\Lambda F^t\bar\Lambda^t)x\\
  =&\bar x^t\bigg(\Phi \big((1+\|w_s\|^2)J\bar J^t-J\bar w_sw_s^t\bar J^t\big)\bar\Phi^t-\Lambda \big((1+\|w_s\|^2)\bar JJ^t-\bar Jw_s\bar w_s^t J^t\big)\bar \Lambda^t\bigg)x\\
  =&(1+\|w_s\|^2)\|\bar J^t\bar\Phi^tx\|^2-|w_s^t\bar J^t\bar\Phi^tx|^2-(1+\|w_s\|^2)\|J^t\bar \Lambda^tx\|^2+|\bar w_s^t J^t\bar \Lambda^tx|^2\\
  =&\|\bar J^t\bar\Phi^tx\|^2+\|w_s\|^2\|\bar J^t\bar\Phi^tx\|^2\sin^2\alpha_s-\| J^t\bar \Lambda^tx\|^2-\|w_s\|^2\|J^t\bar \Lambda^tx\|^2\sin^2\beta_s\\
    \geq&-\|J^t\bar \Lambda^tx\|^2(1+\|w_s\|^2\sin^2\beta_s),
\end{align*}
 when $s$ is small enough. Here $\alpha_s$ (resp. $\beta_s$) is the angle between $w_s$ and $\bar J^t\bar\Phi^tx$ (resp. $J^t\bar \Lambda^tx$).  
Since the local functions $w^i$ on $\cU$ are well defined, we know that $w^i$ and $J$ are bounded on $\cU$ probably after shrinking the disc $\Delta$ and the cover $\{\mathcal{U}_k\}$.   Therefore the form $(\gamma_s)_{1,1}$ has the following positivity  on $U_0=\cU\cap X_0$: for  every $\epsilon>0$,
\begin{align}\label{eqn-pos}
\begin{split}
  &\bar x^t\cdot\frac{1}{(1+\|w_s\|^2)^2}(\Phi F \bar\Phi^t-\Lambda F^t\bar\Lambda^t)\cdot x+\epsilon\|x\|^2\\
    \geq&-\frac{1+\|w_s\|^2\sin^2\beta_s}{(1+\|w_s\|^2)^2} \|J^t\bar \Lambda^tx\|^2+\epsilon\|x\|^2\\
      \geq&-\frac{1+\|w_s\|^2\sin^2\beta_s}{(1+\|w_s\|^2)^2} \|J^t\|_2^2\|\bar \Lambda^t\|_2^2\|x\|^2+\epsilon\|x\|^2>0,
\end{split}
\end{align}
 when $s$ is small enough.

{\bf Step 4.} In this step, we will finish the proof.

 Let $\omega_0$ be a fixed K\"ahler metric on $X_0$ such that 
 $$\omega_0|_{U_0}=\frac{\sqrt{-1}}{2\pi}\sum w_{ij}dz^i\wedge d\bar{z}^j$$
 with $(w_{ij})\ge E$. Since $\cU$ is an arbitrary open set in the finite open cover $\{\cU_k\}_k$ of $ X$, we have that \eqref{eqn-pos} holds true on $X_0$. Therefore, for any $\epsilon>0$, we can choose a small enough $s$ such that $(\gamma_s)_{1,1}+\epsilon \omega_0$ is a positive  (1,1)-form on $X_0$. 
By Step 1, $\gamma_s$ represents the class $c_1(\cL_0)$, thus for every irreducible curve $C\subset X_0$,
 $$\deg(\cL_0|_C)=\int_{C_{\text{reg}}} \gamma_s=\int_{C_{\text{reg}}} (\gamma_s)_{1,1}\geq\int_{C_{\text{reg}}}(-\epsilon \omega_0)=-\epsilon\int_{C_{\text{reg}}}\omega_0,$$
 where $C_{\text{reg}}$ is the  open subset of all regular points on $C$. Since $\epsilon>0$ is small enough, we know that $\deg(\cL_0|_C)\geq0$, and thus $\cL_0$ is nef on $X_0$ by definition.

 \end{proof}

\section{Deformation rigidity}\label{sec5n}
\subsection{Semiampleness} In this subsection, we will show that $\omega_{X_0}$ is semiample. We show the nefness of $\omega_{X_0}$ firstly.
 \begin{prop}\label{sem}  Let $\pi\cln X\to \Delta$ and $S$ satisfy $(*)$.  If  $\omega_{S}$ is nef,  then $\omega_{X_0}$ is also nef.
\end{prop}
\begin{proof}  Let $m$ be any positive integer.  By Proposition \ref{proj}, the morphism $\pi$ is projective, and we can take a $\pi$-very ample line bundle $\cL$  on $X$. Define
 $$\cM^m=\omega^{\otimes m}_X\otimes \mathcal{L},~~\cM^m_{s}=\big(\omega^{\otimes m}_X\otimes \mathcal{L}\big)|_{X_s}.$$

First  we will show that there exists a positive integer $c_m$  such that the line bundles $(\cM^m_{s})^{\otimes c_m}$ satisfy all the conditions of Lemma \ref{form}.

(i)  By \cite[Theorem 0.2]{Siu-1}, there exists a positive integer $c_m$ such that the line bundles $(\cM^m_{s})^{\otimes c_m}$  are very ample for all $s\in\Delta^*$, where the constant $c_m$ depends only on the numbers $n$ and $c^{n-i}_1(\cM^m_{s})c_1^i(\omega_{X_s})~(0\leq i\leq n)$.

(ii) By \cite{Siu2} (or \cite[Theorem 4.1]{FS},\cite{Pa07}), we know that, for each $s\neq 0$,
$$ h^0(X_s,(\cM^m_{s})^{\otimes c_m})=h^0(X_0,(\cM^m_{0})^{\otimes c_m}).$$

So, we  obtain that $(\cM^m_{0})^{\otimes c_m}$ is nef by Lemma  \ref{form}.  Furthermore, $\omega^{\otimes m}_{X_0}\otimes \mathcal{L}_0$ is nef.

Thus, by the arbitrariness of $m\in\mathbb{N}$, we have that $\omega_{X_0}$ is nef.
\end{proof}

Combining with the invariance of plurigenera, we will obtain the semiampleness of $\omega_{X_0}$ by the main result  of  \cite{K1}.

\begin{prop}\label{sem2}  Let $\pi\cln X\to \Delta$ and $S$ satisfy $(*)$.   If  $\omega_{S}$ is semiample, then $\omega_{X_0}$ is also semiample.
\end{prop}
\begin{proof}  Since $\omega_{X_t}$ is semiample for each $t\neq 0$,  we have that $\omega_{X_t}$  is nef. Moreover,
\beq\label{eqn1}
\nu(X_t)=\kappa(X_t)~~\mbox{and }\kappa(X_t)=\kappa(X_0), \quad \mbox{for}\quad t\neq 0,
\eeq
where the first equality follows from \eqref{eqndim} and the second one is from \cite[Theorem 1.2]{RT} (see also \cite{Siu1,Siu2,Pa07}).
Similarly, we know that
\beq\label{eqn2}
\nu(X_t)= \nu(X_0)\ge \kappa(X_0)
\eeq
for all $t\in\Delta$, since  $\omega_{X_0}$ is nef by Proposition  \ref{sem}.
Therefore, by \eqref{eqn1} and \eqref{eqn2}, we get that
$$
\nu(X_0)=\kappa(X_0),
$$
and thus $\omega_{X_0}$ is semiample  by  \eqref{eqndim}.
\end{proof}

\subsection{Proof of Theorem \ref{res1} for $m=1$}\label{sect-stablemnfd}
We are ready to prove the main result Theorem \ref{res1} for the case that the smooth family $\pi$ is over the unit disc $\Delta$. In this case, we may assume that $\pi$ satisfies $(*)$.
\begin{theo}\label{res1-1} Let $\pi\cln X\to \Delta$ and $S$ satisfy $(*)$.  If  $\omega_{S}$ is semiample, then the central fiber $X_0$ is  biholomorphic to $ S$.
\end{theo}
\begin{proof}
By our construction in Section \ref{sec3} we can take a $\pi$-very ample line bundle $\cL$  on $X$. Therefore we obtain a family  $(\pi\cln X\to\Delta,\cL)$ of polarized manifolds.  Since the dualizing sheaf $\omega_S$ is semiample, we know that $\omega_{X_0}$ is also semiample by Proposition \ref{sem2}. So the family $(\pi\cln X\to\Delta,\cL)$ induces a  holomorphic morphism $\mu_\pi$ to the moduli space $P_h$ of polarized manifolds,
$$\mu_\pi\cln \Delta\to P_h,$$
where $h(m)=\chi(\cL^{\otimes m})$ for all $m\in\ZZ$.

 For any $t_1,t_2\in\Delta^*$, $\cL_{t_1}$ and $\cL_{t_2}$ are homological equivalent (see \cite{Fu}) by our construction of $\cL$, and then they  are numerical equivalent. So $\mu_\pi(t_1)=\mu_\pi(t_2)$ by definition. Since the moduli space $P_h$ is separated, the morphism $\mu_\pi$ is constant. Therefore $X_0\cong X_t\cong S$ for all $t\neq0$.
\end{proof}
\begin{remark}
The line bundle $\cL$ constructed in Section \ref{sec3} satisfies that $c_1(\cL_{t_1})=c_1(\cL_{t_2})$. If the Picard group $\mathrm{Pic}(S)$ of $S$ is torsion-free, then $\cL_{t_1}\cong\cL_{t_2}$. Under this condition, Theorem \ref{res1-1} can be derived from \cite[Theorem 2]{MM},
and thus the separatedness of the moduli space $P_h$ is  not necessary in the above proof.  
\end{remark}

\subsection{Proof of Theorem \ref{res1}} Now we give the proof of the main result Theorem \ref{res1}.
\begin{proof}[Proof of Theorem \ref{res1}]
    Let $z$ be  an arbitrary point in $\Delta^m\setminus U$. We may choose a unit disc $\Delta_z\cong \{t\in\CC||t|< 1\}$ passing through $z$ such that the fiber $X_p$ is isomorphic to $S$ for every  $p\in \Delta_z\setminus\{z\}$. Then the restriction morphism $\pi|_{\pi^{-1}(\Delta_z)}\cln \pi^{-1}(\Delta_z)\to \Delta_z$ satisfies all the conditions of Theorem \ref{res1-1}, thus $X_z\cong S$ by  Theorem \ref{res1-1}. 
\end{proof}

\section{Isotriviality for smooth families}\label{sec6n}

In this section, we will give some applications of our main result to the isotriviality of smooth families. To this end, we will give a detailed description for the reader's convenience.

 Let $f\colon X\to B$ be a smooth  projective family between smooth projective varieties.  A {\it minimal closed field of definition} of $f$ is a minimal element with respect to the inclusion of the set of all the algebraically closed fields $K$ contained in $\overline{\CC(B)}$  which satisfy the following condition: 
  there  exists a finitely generated extension $L$ of $K$ such that
  $$Q(L\otimes_K\overline{\CC(B)})\cong Q(\CC(X)\otimes_{\CC(B)}\overline{\CC(B)})~~\mathrm{over}~~\overline{\CC(B)},$$
  where $Q$ denotes the fraction field.

  This condition is equivalent to the following one (see \cite{K0}): there is a projective family $f^!\colon X^!\to B^!$ with $\overline{\CC(B^!)}=K$, a generically finite and generically surjective morphism $\pi\cln\bar{B}\to B$ and a generically surjective morphism $\rho\cln \bar{B}\to B^!$ such that the induced projective family $\bar f\cln \bar X\to\bar B$ from $f$ by $\pi$ is birationally equivalent to that from $f^!$ by $\rho$ as illustrated in the following commutative diagram
   $$\xymatrix{
  X \ar[d]_{f} & \bar X \ar[l] \ar[r] \ar[d]_{\bar f} & X^!\ar[d]^{f^!}\\
   B             & \bar B \ar[l]_{\pi} \ar[r]^{\rho} &   B^!.                }
                $$

   A minimal closed field of definition of $f$ is a priori not unique. However,  it is unique in the case where the geometric generic fiber $X_{\eta}$ has a good minimal model. The {\it variation} $\mathrm{Var}(f)$ of $f$ is defined by
   $$\mathrm{Var}(f)=\mathrm{trans.deg.}_\CC K$$
   for an arbitrarily chosen minimal closed field of definition $K$ of $f$. In  order to avoid the ambiguity in the definition, we may take the minimum of those values given by minimal closed fields of definition. Roughly speaking, $\mathrm{Var}(f)$ is the number of moduli of fibers of $f$ in the sense of birational geometry.  Note that $0\leq \mathrm{Var}(f)\leq \dim B$.

So, we know that $f$ is birationally isotrivial is equivalent to $\mathrm{Var}(f)=0$ by definition.

\begin{theo}\label{g-gen-intr}
  Let $f\colon X\to B$ be a smooth projective family between smooth projective varieties. Suppose that the generic fiber of $f$  is  good minimal. If $f$ is birationally isotrivial, then $f$ is isotrivial.
\end{theo}
\begin{proof}
Since the generic fiber of $f$ is a smooth good minimal algebraic variety over $\CC(B)$, by \cite[Lemma 7.1]{K0},
 there are nonsingular projective varieties $\bar B$ and $B^!$, a non-empty open subset $U$ of $\bar B$,  a generically finite and surjective morphism $\bar \tau\cln \bar B\to B$, a surjective morphism $\bar\psi\cln \bar B\to B^!$, and a normal family $f^!\cln X^!\to B^!$ which satisfy that  $$X\times_{B}U\cong X^!\times_{B^!}U,$$ and the generic fiber  of $f^!$ is a good minimal algebraic variety over $\CC(B^!)$.
    \[\begin{CD}
X @<<<  \bar X @>>> X^!\\
@V {f} VV              @V {\bar f} VV         @V{\bar f^!} VV \\
B @<{\bar\tau}<< \bar B @>{\bar\psi}>> B^!
\end{CD}
\]
Furthermore, by \cite[Corollary 7.3]{K0},   the minimal closed field of definition of $f$ coincides with $\overline{\CC(B^!)}$.  Now  $f$ is birationally isotrivial, that is, $\mathrm{Var}(f)=0$, so there is a non-empty open subset $V$ of $B$  and an \'etale covering $V'\to V$ such that
$$X\times_VV'\cong F\times V'$$
 for some projective variety $F$.

  Thus all the fibers of $f$ over $V$ are isomorphic to $F$, and then all the fibers of $f$ are isomorphic by Theorem \ref{res1}.
\end{proof}

\begin{coro}\label{g-gen-2}
  Let $f\colon X\to B$ be a smooth projective family between smooth projective varieties. Suppose that either $F$ is of general type or $\dim F\leq3$. If the generic fiber of $f$  is minimal and $\deg(\det f_*\omega_{X/B}^{\otimes k})=0$ for all $k>0$, then    $f$ is isotrivial.
\end{coro}
\begin{proof}
  Since $\deg(\det f_*\omega_{X/B}^{\otimes k})=0$ for each $k>0$,  we have that $\mathrm{Var}(f)=0$ by \cite[Theorem 1.1]{K0}. Then the result is directly from Theorem \ref{g-gen-intr}.
\end{proof}

\begin{coro}\label{g-gen-intr-coro}
  Let $f\colon X\to B$ be a smooth  projective family between smooth projective varieties where $B$ is a  curve. Suppose that the generic fiber $F$ of $f$ is minimal  of general type and $\deg(\det f_*\omega_{X/B}^{\otimes k})=0$ for some sufficiently large $k$. Then $f$ is isotrivial.
\end{coro}
\begin{proof}

 Suppose $f$ is not isotrivial,   $\mathrm{Var}(f)\neq0$ by Theorem \ref{g-gen-intr} and thus $\mathrm{Var}(f)=\dim B=1$. Then by \cite[Theorem on p.363]{Kol} $ f_*\omega_{ X/B}^{\otimes \eta}$ is big for some $\eta>0$. Since $F$ is of general type, by the base-point-free theorem, there exists an integer $m_0$ such that $|\omega_{X_t}^{\otimes m}|$ is base point free  for $m>m_0$. Thus $f^*f_*\omega_{X/B}^{\otimes k}\to \omega_{X/B}^{\otimes k}$ is surjective.  By \cite[Theorem 0.1]{EW}, this implies that  $\det f_*\omega_{ X/B}^{\otimes k}$ is big for sufficiently large $k$,  which leads to a contradiction.
 \end{proof}



\bibliographystyle{amsplain}

{\small {School of Mathematics, Hunan University, China

 {\it E-mail address}: mulin@hnu.edu.cn}

 {School of Mathematical Sciences, Dalian University of Technology,  China

 {\it E-mail address}: xlliu1124@dlut.edu.cn}}

\end{document}